\documentclass[a4paper,11pt]{article}
\usepackage{amsfonts}
\usepackage{amstext}
\usepackage{amsthm}
\usepackage{amsmath}
\usepackage{amssymb}
\usepackage{mathtools}
\usepackage{latexsym}
\usepackage{graphicx}
\usepackage{caption}
\usepackage{enumitem}
\usepackage[utf8]{inputenc}
\usepackage[hidelinks]{hyperref}
\usepackage{rotating}
\newcommand{\R}{\mathbb{R}}
\newcommand{\N}{\mathbb{N}}

\newtheorem{teor}{Theorem}[section]

\newtheorem{defi}{Definition}[section]
\newtheorem{lema}{Lemma}[section]
\newtheorem{cor}{Corollary}[section]

\newcommand{\n}{\noindent}

\title{}
\author{}
\date{}

\begin{document}

\title{On the characterization of the Dirichlet and Fu\v cík spectra of\\ the one-dimensional anisotropic $p$-Laplace operator
\footnote{2020 Mathematics Subject Classification: 34B15; 34L05; 34L10; 34L15; 34L30}
\footnote{Key words: Anisotropic $p$-Laplacian, first eigenvalue, one-sided Poincaré inequality, Dirichlet spectrum}
}

\author{\textbf{Raul Fernandes Horta \footnote{\textit{E-mail addresses}: raul.fernandes.horta@gmail.com
 (R. F. Horta)}}\\ {\small\it Departamento de Matem\'{a}tica,
Universidade Federal de Minas Gerais,}\\ {\small\it Caixa Postal
702, 30123-970, Belo Horizonte, MG, Brazil}\\
\textbf{Marcos Montenegro \footnote{\textit{E-mail addresses}:
montene@mat.ufmg.br (M. Montenegro)}}\\ {\small\it Departamento de Matem\'{a}tica,
Universidade Federal de Minas Gerais,}\\ {\small\it Caixa Postal
702, 30123-970, Belo Horizonte, MG, Brazil}}

\date{}{%{\it Preprint - December 12, 2009}}

\maketitle

\markboth{abstract}{abstract}
\addcontentsline{toc}{chapter}{abstract}

\hrule \vspace{0,2cm}

\n {\bf Abstract}

The paper is concerned with the Dirichlet spectrum $\Lambda^{a,b}_p(0,L)$ of the anisotropic $p$-Laplace operator $- \Delta^{a,b}_{p}$ on an interval $(0,L)$ where

\[
\Delta^{a,b}_p u:= \left(a^{p}[(u')^{+}]^{p-1}-b^{p}[(u')^{-}]^{p-1}\right)',\footnote{Notation: $(u')^\pm = \max\{\pm u', 0\}$} \ \ a, b > 0.
\]
The set $\Lambda^{a,b}_p(0,L)$ and the respective eigenfunctions are completely characterized for $a \neq b$ in terms of the corresponding ones within the isotropic context. As an interesting application, we derive a new optimal Poincaré inequality that is stronger than the classical counterpart. The leading ideas are based on glue arguments of conveniently modified eigenfunctions and maximum type principles. More generally, our approach allows to characterize the Fu\v cík spectrum $\Sigma^{a,b}_p(0,L)$ of $- \Delta^{a,b}_{p}$ on $(0,L)$ and mainly the corresponding solutions. All results are novelty even for the nonlinear operator $\Delta^{a,b}_2$.

\vspace{0.5cm}
\hrule\vspace{0.2cm}

\section{Introduction and main theorems}

Quasilinear elliptic equations involving the $p$-Laplace operator arise in a variety of situations such as fluid flows through some types of porous media, singular solutions for the Emden-Fowler and Einstein-Yang-Mills equations, the Lane-Emden equations for equilibrium configurations of gaseous stellar objects, among many others, see for instance Chapter 14 of the book \cite{Ky}.

In this paper we carry out a complete study on the spectrum of such operators in dimension one involving asymmetric conductivity. More specifically, given numbers $p > 1$ and $a, b, L > 0$ with $a \neq b$, let $\Delta^{a,b}_p$ be the {\it anisotropic} $p$-Laplace operator defined weakly for functions $u \in W_0^{1,p}(0,L)$ as

\[
\langle \Delta^{a,b}_p u, \varphi \rangle = - \int_0^L \left(a^{p}[(u'(t))^{+}]^{p-1}-b^{p}[(u'(t))^{-}]^{p-1}\right) \varphi'(t)\, dt
\]
for every $\varphi \in W_0^{1,p}(0,L)$, in short

\[
\Delta^{a,b}_p u = \left(a^{p}[(u')^{+}]^{p-1}-b^{p}[(u')^{-}]^{p-1}\right)'.
\]

We consider the associated boundary value problems

\begin{equation} \label{DS}
\left\{
\begin{array}{lllllrlllr}
- \Delta^{a,b}_p u = \lambda \vert u \vert^{p-2} u & {\rm in} & (0, L), \\
u(0) = u(L) = 0
\end{array}\right.
\end{equation}
and

\begin{equation} \label{FS}
\left\{
\begin{array}{lllllrlllr}
- \Delta^{a,b}_p u = \mu (u^+)^{p-1} - \nu (u^-)^{p-1} & {\rm in} & (0, L), \\
u(0) = u(L) = 0,
\end{array}\right.
\end{equation}
where $\lambda, \mu, \nu \in \R$ and $u^\pm = \max\{\pm u, 0\}$.

Corresponding to \eqref{DS} and \eqref{FS}, we have the respective Dirichlet and Fu\v cík spectra:

\begin{eqnarray*}
&& \Lambda^{a,b}_p(0,L) := \{ \lambda \in \R:\ \eqref{DS}\ \text{admits a nontrivial weak solution}\, u \in W_0^{1,p}(0,L)\}, \\
&& \Sigma^{a,b}_p(0,L) := \{ (\mu, \nu) \in \R^2:\ \eqref{FS}\ \text{admits a nontrivial weak solution}\, u \in W_0^{1,p}(0,L)\}.
\end{eqnarray*}
The elements of $\Lambda^{a,b}_p(0,L)$ are called Dirichlet eigenvalues of $- \Delta^{a,b}_p$.

When $a = b$, the operator $\Delta^{a,a}_p$ coincides with $a^p \Delta_p$, where $\Delta_p u := \left(\vert u' \vert^{p-2} u'\right)'$ denotes the one-dimensional isotropic $p$-Laplace operator. Notice also that both sets $\Lambda_p(0,L):= a^{-p}\, \Lambda^{a,a}_p(0,L)$ and $\Sigma_p(0,L):= a^{-p}\, \Sigma^{a,a}_p(0,L)$ are entirely known. Precisely, they are given by

\begin{eqnarray*}
&& \bullet\ \ \Lambda_p(0,L) = \{ \lambda_{k,p}(0,L):\, k \geq 1\}, \\
&& \bullet\ \ \Sigma_p(0,L) = \left( \{ \lambda_{1,p}(0,L) \} \times \R \right) \cup \left( \R \times \{ \lambda_{1,p}(0,L) \} \right) \cup \Sigma^e_p(0,L) \cup \Sigma^{o,1}_p(0,L) \cup \Sigma^{o,2}_p(0,L),
\end{eqnarray*}
where

\[
\lambda_{k,p}(0,L) = (p-1) \left( \frac{k \pi_p}{L} \right)^p = k^p \lambda_{1,p}(0,L)
\]
with

\[
\pi_p := 2 \int_0^1 (1 - t^p)^{-1/p} dt
\]
and

\begin{eqnarray*}
&& \Sigma^e_p(0,L):= \{(\mu, \nu) \in \R_+^2:\, \mu^{-1/p} + \nu^{-1/p} = 2 \lambda_{2k,p}(0,L)^{-1/p},\ k \geq 1\},\\
&& \Sigma^{o,1}_p(0,L):= \{(\mu, \nu) \in \R_+^2:\, k \mu^{-1/p} + (k + 1) \nu^{-1/p} = (2k + 1) \lambda_{2k+1,p}(0,L)^{-1/p},\ k \geq 1\}, \\
&& \Sigma^{o,2}_p(0,L):= \{(\mu, \nu) \in \R_+^2:\, (k+1) \mu^{-1/p} + k \nu^{-1/p} = (2k + 1) \lambda_{2k+1,p}(0,L)^{-1/p},\ k \geq 1\}.
\end{eqnarray*}
For the one-dimensional Dirichlet spectrum of $- \Delta_p$ we refer to \cite{dPEM} and for the associated Fu\v cík spectrum to \cite{D,F} for $p = 2$ and to \cite{Dra} for $p \neq 2$.

As it is well known, there is a vast literature on miscellaneous problems connected to \eqref{DS} and \eqref{FS} in the case $a = b$, both in dimension 1 and higher dimensions. Most contributions can be found in a long list of old and new references. For some more closely related works, we refer to \cite{A, AT, AriC, AriCCuGo, BWZ, CudeFG, CuG, DD, deFG, dPDM, DraGT, DraS1, DraS2, DraS3, DraT, KL, LL, L, MaT, M, S, S1, T}.

On the other hand, fairly little is known about spectra of $- \Delta^{a,b}_p$ when $a \neq b$, even for $p = 2$ in which $\Delta^{a,b}_2$ is a nonlinear elliptic operator. Our main goal here is to determine explicitly both spectral sets as well as to characterize all corresponding solutions for any $p > 1$ and $a, b > 0$. In particular, we establish an one-to-one correspondence between the structures of the anisotropic and isotropic spectra, showing that many spectral properties are not affected by the asymmetry of the operator $\Delta^{a,b}_p$.

The bridge between the asymmetric and symmetric settings will be constructed by means of two fundamental ingredients:

\begin{itemize}
\item[(i)] The explicit value of the first Dirichlet eigenvalue of $-\Delta^{a,b}_p$ as well as the classification of its eigenfunctions (Theorem \ref{T1});
\item[(ii)] The simplicity of the zeroes of any nontrivial weak solution of \eqref{FS} (Theorem \ref{T4}).
\end{itemize}
Due to the presence of asymmetry, the proofs of (i) and (ii) demand alternative arguments adapted to the current context. Particularly, the approach of (i) involves the construction of some appropriate glues of modified eigenfunctions and the recurrent use of maximum type principles. In addition, it is worth mentioning that the explicit knowledge of the first Dirichlet eigenvalue in the isotropic case, as well as of the key properties satisfied by its eigenfunctions, plays an essential role in the proof of the ingredient (i).

A well-known property of the first Dirichlet eigenvalue $\lambda_{1,p}(0,L)$ of $- \Delta_p$ is its simplicity, that is, the set of all eigenfunctions associated to $\lambda_{1,p}(0,L)$ is generated by a positive eigenfunction $\varphi_p$:

\[
\{t \varphi_p:\, t \neq 0,\ \varphi_p \in C^1[0,L] \ \ {\rm and}\ \ \varphi_p > 0\ {\rm in}\ (0, L) \}.
\]
Note that this set is the union of two collinear half-lines in $C^1[0,L]$. Within the anisotropic environment, the appropriate notion of simplicity for Dirichlet eigenvalues of the operator $-\Delta^{a,b}_p$ is as follows:

\begin{defi} \label{simple}
A Dirichlet eigenvalue of $-\Delta^{a,b}_p$ is said to be simple in the generalized sense when $a \neq b$, if the set of corresponding eigenfunctions is the union of two non-collinear half-lines in $C^1[0,L]$, one generated by a positive function and other one by a negative function in $(0,L)$.
\end{defi}

For each $p > 1$ and $a, b, L > 0$, consider the number

\[
\lambda^{a,b}_{1,p}(0,L):= \inf \left\{ \int_0^L a^{p}[(u'(t))^{+}]^{p} + b^{p}[(u'(t))^{-}]^{p}\, dt:\ u \in W_0^{1,p}(0,L),\ \Vert u \Vert_p = 1 \right\}.
\]
From the fact that $\lambda_{1,p}(0,L)$ is positive, it is clear that $\lambda^{a,b}_{1,p}(0,L)$ is also positive and, in addition, it follows by direct minimization, with the aid of the Poincaré inequality, that the infimum is attained for some $u \in W_0^{1,p}(0,L)$. In particular, $\lambda^{a,b}_{1,p}(0,L)$ is the smallest Dirichlet eigenvalue of $-\Delta^{a,b}_p$ and $u$ is a corresponding eigenfunction. Furthermore, by definition, $\lambda^{a,b}_{1,p}(0,L)$ is the best constant of the asymmetric Poincaré inequality

\[
\lambda^{a,b}_{1,p}(0,L) \int_0^L \vert u(t) \vert^p\, dt \leq \int_0^L a^{p}[(u'(t))^{+}]^{p} + b^{p}[(u'(t))^{-}]^{p}\, dt
\]
for every $u \in W_0^{1,p}(0,L)$. Under this point of view, the heuristic intuition based on similarity of roles played by the terms $(u')^{+}$ and $(u')^{+}$ perhaps would suggest that the value of $\lambda^{a,b}_{1,p}(0,L)$ should be

\[
\displaystyle \left( \frac{a^p + b^p}{2}\right) \lambda_{1,p}(0,L).
\]
However, its computation does not seem trivial by far and, to our surprise, it is strictly smaller than the one mentioned above. Part of our first theorem exhibits the precise value of $\lambda^{a,b}_{1,p}(0,L)$.

\begin{teor} \label{T1}
Let $p > 1$ and $L > 0$. For any $a, b > 0$, we have:

\begin{itemize}
\item[(I)] $\lambda_{1,p}^{a,b}(0,L) = \displaystyle \left(\frac{a+b}{2}\right)^{p}\lambda_{1,p}(0,L)$;
\item[(II)] $u \in W_0^{1,p}(0,L)$ is an eigenfunction of $-\Delta^{a,b}_p$ associated to $\lambda^{a,b}_{1,p}(0,L)$ if and only if either

\[
u(t) = c\varphi_{p, L}^{a,b,+}(t) := c\varphi_{p}^{a,b,+}\left(\frac{t}{L}\right)\ \ {\rm or}\ \ u(t) = c\varphi_{p, L}^{a,b,-}(t) := c\varphi_{p}^{a,b,-}\left(\frac{t}{L}\right)
\]
for some constant $c>0$, where

\begin{equation*}
\varphi_{p}^{a,b,+}(t) =
\begin{cases}
\begin{aligned}
     & \varphi_{p}\left(\frac{t}{2t_{0}}\right)\ \ \text{if} \ t \in (0,t_{0}), \\
     & \varphi_{p}\left(\frac{t + 1 - 2 t_{0}}{2(1-t_{0})}\right)\ \ \text{if} \ t \in [t_{0},1),
\end{aligned}
\end{cases}
\end{equation*}
\begin{equation*}
\varphi_{p}^{a,b,-}(t) =
\begin{cases}
\begin{aligned}
     & -\varphi_{p}\left(\frac{t}{2(1 - t_{0})}\right)\ \ \text{if} \ t \in (0,1 - t_{0}), \\
     & -\varphi_{p}\left(\frac{t + 2 t_{0} - 1}{2 t_{0}}\right)\ \ \text{if} \ t \in [1 - t_{0},1).
\end{aligned}
\end{cases}
\end{equation*}
Here, $\varphi_{p} \in W^{1,p}_0(0,1)$ denotes the principal eigenfunction associated to $\lambda_{1,p}(0,1)$ normalized by $\Vert \varphi_{p} \Vert_p = 1$ and

\[
t_{0} = \frac{a}{a+b}.
\]
In particular, $\lambda_{1,p}^{a,b}(0,L)$ is an simple eigenvalue in the generalized sense, in other words, the eigenspace of $\lambda^{a,b}_{1,p}(0,L)$ is formed by two non-collinear half-lines if and only if $a \neq b$.
\end{itemize}
\end{teor}
This result will play an important role in the complete characterization of the Dirichlet and Fu\v cík spectra of $-\Delta^{a,b}_p$ and corresponding solutions for any $a, b > 0$. It also implies the following one-sided Poincaré inequalities by letting $a \rightarrow 0$ with $b = 1$ (or $b \rightarrow 0$ with $a = 1$) in the asymmetric Poincaré inequality:

\begin{cor} \label{C1}
Let $p > 1$ and $L > 0$. It holds that

\[
2^{-p} \lambda_{1,p}(0,L) \int_0^L \vert u(t) \vert^p\, dt \leq \int_0^L [(u'(t))^{\pm}]^{p}\, dt
\]
for every $u \in W_0^{1,p}(0,L)$ and, moreover, both inequalities are optimal.
\end{cor}

\begin{proof}
It suffices to show optimality for the inequality with the sign $+$ since $(-u')^{+} = (u')^{-}$. Let $\lambda^+_{1,p}(0,L)$ be the associated best constant. From the above inequality, clearly $\lambda^+_{1,p}(0,L) \geq 2^{-p} \lambda_{1,p}(0,L)$. On the other hand, for an extremal function $u_b \in W_0^{1,p}(0,L)$ of $\lambda_{1,p}^{1,b}(0,L)$ normalized by $\Vert u_b \Vert_p = 1$ where $b > 0$, we have

\[
2^{-p} \lambda_{1,p}(0,L) \leq \int_0^L [(u'_b(t))^{+}]^{p}\, dt \leq \int_0^L [(u'_b(t))^{+}]^{p} + b^{p}[(u'_b(t))^{-}]^{p}\, dt = \lambda_{1,p}^{1,b}(0,L).
\]
Letting now $b \rightarrow 0$ and using Theorem \ref{T1}, we get

\[
2^{-p} \lambda_{1,p}(0,L) \leq \liminf_{b \rightarrow 0} \int_0^L [(u'_b(t))^{+}]^{p}\, dt \leq \limsup_{b \rightarrow 0} \int_0^L [(u'_b(t))^{+}]^{p}\, dt = 2^{-p} \lambda_{1,p}(0,L),
\]
so that $\lambda^+_{1,p}(0,L) = 2^{-p} \lambda_{1,p}(0,L)$.
\end{proof}

A natural and probably delicate question is whether the stronger inequalities in Corollary \ref{C1} admit any extremal functions. Although this is an open problem, it seems reasonable to expect that no extremal function exists in the space $W_0^{1,p}(0,L)$, once the function $u(t):= \varphi_p(t/2)$ for $t \in [0,L]$ clearly don't belong to $W_0^{1,p}(0,L)$ and, due to the symmetry of $\varphi_p$, satisfies the equality

\[
\int_0^L [(u'(t))^{+}]^{p}\, dt = \int_0^L 2^{-p} \varphi'_p(t/2)^{p}\, dt = 2^{1-p} \int_0^{L/2} \varphi'_p(t)^{p}\, dt = 2^{-p} \int_0^{L} \varphi'_p(t)^{p}\, dt
\]

\[
= 2^{-p} \lambda_{1,p}(0,L) \int_0^L \vert \varphi_p(t) \vert^p\, dt = 2^{1-p} \lambda_{1,p}(0,L) \int_0^{L/2} \vert \varphi_p(t) \vert^p\, dt = 2^{-p} \lambda_{1,p}(0,L) \int_0^L \vert u(t) \vert^p\, dt.
\]

In order to state two other important consequences of the first theorem, consider the function $\phi_{p}^{a,b} \colon \mathbb{R} \to \mathbb{R}$ given by
\begin{equation*}
\phi_{p}^{a,b}(t) =
\begin{cases}
\begin{aligned}
     &\varphi_{p}^{a,b,+}(t)\ \ \text{if} \ t \in [0,1), \\
     &\varphi_{p}^{a,b,-}(t-1)\ \ \text{if} \ t \in [1,2]
\end{aligned}
\end{cases}
\end{equation*}
and extended periodically as $\phi_{p}^{a,b}(t+2) = \phi_{p}^{a,b}(t)$ for every $t \in \mathbb{R}$, where $\varphi_{p}^{a,b,+}$ and $\varphi_{p}^{a,b,-}$ are defined in Theorem~\ref{T1}.

\begin{teor} \label{T2}
Let $p > 1$ and $L > 0$. For any $a, b > 0$, we have:

\begin{itemize}
\item[(I)] $\Lambda^{a,b}_p(0,L) = \displaystyle \left( \frac{a + b}{2} \right)^p \Lambda_p(0,L)$, which means that

\[
\Lambda^{a,b}_p(0,L) = \left\{\lambda^{a,b}_{k,p}(0,L) := \left( \frac{k(a + b)}{2} \right)^p \lambda_{1,p}(0,L):\ k \geq 1 \right\};
\]
\item[(II)] $u \in W_0^{1,p}(0,L)$ is an eigenfunction of $-\Delta^{a,b}_p$ associated to $\lambda^{a,b}_{k,p}(0,L)$ if and only if either $u(t) = c \phi_{k,p,L}^{a,b}(t) := c\phi_{p}^{a,b}\left(\frac{kt}{L}\right)$ or $u(t) = c \phi_{k,p,L}^{a,b}(t) := c\phi_{p}^{a,b}\left(\frac{kt}{L}-1\right)$ for some constant $c > 0$.
\end{itemize}
\end{teor}

More generally, we have the following extension of Theorem \ref{T2}:

\begin{teor} \label{T3}
Let $p > 1$ and $L > 0$. For any $a, b > 0$, we have:

\begin{itemize}
\item[(I)] $\Sigma^{a,b}_p(0,L) = \displaystyle \left( \frac{a + b}{2} \right)^p \Sigma_p(0,L)$;

\item[(II)] $u \in W_0^{1,p}(0,L)$ is a nontrivial weak solution of \eqref{FS} corresponding to a couple $(\mu, \nu) \in \Sigma^{a,b}_p(0,L)$ if and only if there are positive integers $P = P(\mu, \nu)$ and $N = N(\mu, \nu)$ and an interspersed cover of $[0,L]$ by open intervals $\{I_i\}^P_{i=1}$ and $\{J_j\}^N_{j=1}$ of measures $\vert I_i \vert = l_\mu$ and $\vert J_j \vert = l_\nu$, with

\[
l_\mu := \left( \frac{\lambda^{a,b}_{1,p}(0,1)}{\mu} \right)^{1/p}\ \ {\rm and}\ \ \ l_\nu := \left( \frac{\lambda^{a,b}_{1,p}(0,1)}{\nu} \right)^{1/p},
\]
such that $u(t) = c\Phi_{(\mu, \nu), p, L}^{a,b}\left( t \right)$ for some constant $c>0$, where

\begin{equation*}
\Phi_{(\mu, \nu), p, L}^{a,b}\left( t \right) =
\begin{cases}
\begin{aligned}
     & l_\mu\, \varphi_{p}^{a,b,+} \circ \gamma_i(t)\ \ \text{if} \ t \in \bar{I}_i, \\
     & l_\nu\,  \varphi_{p}^{a,b,-} \circ \pi_j(t)\ \ \text{if} \ t \in \bar{J}_j,
\end{aligned}
\end{cases}
\end{equation*}
where $\gamma_i$ and $\pi_j$ represent the increasing affine bijections from $\bar{I}_i$ and $\bar{J}_j$ onto $[0,1]$, respectively. Here, for an interspersed cover $\{I_i\}^P_{i=1}$ and $\{J_j\}^N_{j=1}$ of $[0,L]$ by open intervals, we mean that

\begin{itemize}
\item[(a)] $\bar{I}_i \cap \bar{I}_j = \emptyset = \bar{J}_i \cap \bar{J}_j$ for any $i \neq j$;
\item[(b)] $I_i \cap J_j = \emptyset$ for any $i, j$;
\item[(c)] $I_i \leq I_{i + 1}$ and $J_j \leq J_{j + 1}$ for any $i, j$;
\item[(d)] $\left( \cup^P_{i = 1} \bar{I}_i \right) \cup \left( \cup^N_{j = 1} \bar{J}_j \right) = [0, L]$,
\end{itemize}
where the notation $A \leq B$ means that $x \leq y$ for any $x \in A$ and $y \in B$.
\end{itemize}
\end{teor}

\section{Proof of Theorem \ref{T1}}

For convenience, module a scaling of the domain, we assume that $L = 1$. Let $a,b > 0$ be any fixed numbers. By direct minimization theory, there is $u \in W^{1,p}_{0}(0,1)$ with $\lVert u \rVert_{p} = 1$ such that

\[
\int_{0}^{1} a^p \left( u'(t)^+ \right)^p + b^p \left( u'(t)^- \right)^p dt = \lambda_{1,p}^{a,b}(0,1).
\]
It is easy to see that $u$ is a weak solution of

\begin{equation}\label{pde}
-\left(a^{p}[(u')^{+}]^{p-1}-b^{p}[(u')^{-}]^{p-1}\right)' = \lambda_{1,p}^{a,b}(0,1)|u|^{p-2}u.
\end{equation}
To be more precise, for any $\varphi \in W^{1,p}_{0}(0,1)$, $u$ satisfies

\[
\int_{0}^{1}\left(a^{p}[(u'(t))^{+}]^{p-1}-b^{p}[(u'(t))^{-}]^{p-1}\right)\varphi'(t) dt =  \lambda_{1,p}^{a,b}(0,1)\int_{0}^{1}|u(t)|^{p-2}u(t)\varphi(t) dt.
\]
Standard regularity theory ensures us that $u \in C^1[0,1]$. Notice also that $u$ has definite sign in $(0,1)$. Otherwise, there would be a number $\overline{t} \in (0,1)$ such that $u \in W^{1,p}_{0}(0,\overline{t})\cap W^{1,p}_{0}(\overline{t},1)$, and thus

\begin{eqnarray*}
\lambda_{1,p}^{a,b}(0,1) &=& \int_{0}^{1} a^p \left( u'(t)^+ \right)^p + b^p \left( u'(t)^- \right)^p dt\\
&=& \int_{0}^{\overline{t}} a^p \left( u'(t)^+ \right)^p + b^p \left( u'(t)^- \right)^p dt + \int_{\overline{t}}^{1} a^p \left( u'(t)^+ \right)^p + b^p \left( u'(t)^- \right)^p dt\\
&=& \lambda_{1,p}^{a,b}(0,\overline{t}) \int_{0}^{\overline{t}} |u(t)|^p dt + \lambda_{1,p}^{a,b}(\overline{t}, 1) \int_{\overline{t}}^{1} |u(t)|^p dt \\
&\geq& \min\{\lambda_{1,p}^{a,b}(0,\overline{t}), \lambda_{1,p}^{a,b}(\overline{t}, 1)\},
\end{eqnarray*}
which contradicts the strict monotonicity $\lambda_{1,p}^{a,b}(0,1) < \lambda_{1,p}^{a,b}(0,\overline{t})$ and $\lambda_{1,p}^{a,b}(0,1) < \lambda_{1,p}^{a,b}(\overline{t}, 1)$ with respect to the domain. This latter follows readily from the positivity of $\lambda_{1,p}^{a,b}(0,1)$, invariance property by translation and the scaling property

\[
\lambda_{1,p}^{a,b}(0, L) = \frac{1}{L^p} \lambda_{1,p}^{a,b}(0,1)
\]
for every $L > 0$.

From now on assume that $u \geq 0$ (the case $u \leq 0$ is analogous). By Harnack inequality (see \cite{Tr}), we have $u > 0$ in $(0,1)$. Since $u(0) = 0$ and $u(1) = 0$, by the Hopf´s lemma, we know that $u'(0) >  0$ and $u'(1) < 0$. Let $t_0 := \min\{t \in [0,1]: u'(t) = 0\}$ and $t_1 := \max\{t \in [0,1]: u'(t) = 0\}$. Clearly, $0 < t_0 \leq t_1 < 1$, $u'(t) > 0$ for every $t \in (0,t_0)$ and $u'(t) < 0$ for every $t \in (t_1, 1)$. We assert that $t_0 = t_1$. Assume by contradiction that $t_0 < t_1$ and consider the following functions

\begin{equation*}
w_1(t) =
\begin{cases}
\begin{aligned}
     & u(t)\ \ \text{if} \ t \in (0,t_{0}), \\
     & u(2t_{0}-t)\ \ \text{if} \ t \in [t_{0},2t_{0}),
\end{aligned}
\end{cases}
\end{equation*}
\begin{equation*}
w_2(t) =
\begin{cases}
\begin{aligned}
     & u(2t_{1}-t)\ \ \text{if} \ t \in (2t_{1}-1,t_{1}), \\
     & u(t)\ \ \text{if} \ t \in [t_{1},1).
\end{aligned}
\end{cases}
\end{equation*}
Note that $w_1 > 0$ in $(0,2 t_0)$ and $w_2 > 0$ in $(2 t_1 - 1,1)$ and that $w_1 \in W^{1,p}_{0}(0,2t_{0})$ and $w_2 \in W^{1,p}_{0}(2t_{1}-1,1)$. Also, since $u$ is a weak solution of \eqref{pde} and $u'(t_0) = u'(t_1) = 0$, by Lemma \ref{L1}, each of them satisfies in the weak sense the respective equations

\[
-(|w_1'|^{p-2}w_1)' = \frac{\lambda_{1,p}^{a,b}(0,1)}{a^{p}}|w_1|^{p-2}w_1
\]
and

\[
-(|w_2'|^{p-2}w_2)' = \frac{\lambda_{1,p}^{a,b}(0,1)}{b^{p}}|w_2|^{p-2}w_2,
\]
which implies that $w_1$ is a principal eigenfunction of the $p$-Laplace operator associated to $\lambda_{1,p}(0,2t_{0})$ and $w_2$ is an one associated to $\lambda_{1,p}(2t_{1}-1, 1)$. Hence, by spectral properties of $- \Delta_p$, we have

\[
\lambda_{1,p}(0,2t_{0}) = \frac{\lambda_{1,p}^{a,b}(0,1)}{a^{p}}
\]
and

\[
\lambda_{1,p}(2t_{1}-1, 1) = \frac{\lambda_{1,p}^{a,b}(0,1)}{b^{p}},
\]
which yield

\[
a^{p}\lambda_{1,p}(0,2t_{0}) = b^{p}\lambda_{1,p}(2t_{1}-1, 1)
\]
or equivalently

\begin{equation}\label{lambdaab1}
 \frac{a^{p}}{(2t_{0})^{p}} \lambda_{1,p}(0,1) = \frac{b^{p}}{(2(1-t_{1}))^{p}} \lambda_{1,p}(0,1).
\end{equation}
Since we are assuming by contradiction that $t_0 < t_1$, the above equality yields $t_{1} > \frac{a}{a+b}$ and so

\begin{equation} \label{lambdaab2}
\lambda_{1,p}^{a,b}(0,1) =  \frac{b^{p}}{(2(1 - t_{1}))^{p}}\lambda_{1,p}(0,1) > \frac{b^{p}}{\left(2\left(1 - \frac{a}{a+b}\right)\right)^{p}}\lambda_{1,p}(0,1)  = \left(\frac{a+b}{2}\right)^{p}\lambda_{1,p}(0,1).
\end{equation}
Now let $\varphi_{p} \in W_0^{1,p}(0,1)$ be as in the statement and let $t^{*} =\frac{a}{a+b}$. Define the positive function

\begin{equation}\label{lambdaab3}
u(t) =
\begin{cases}
\begin{aligned}
     & \varphi_{p}\left(\frac{t}{2t^{*}}\right)\ \ \text{if} \ t \in (0,t^{*}), \\
     & \varphi_{p}\left(\frac{t + 1 - 2t^{*}}{2(1 - t^{*})}\right)\ \ \text{if} \ t \in [t^{*},1).
\end{aligned}
\end{cases}
\end{equation}
It is clear that $u \in W^{1,p}_{0}(0,1)$ and notice that

\begin{align*}
\int_{0}^{1} |u|^{p} dt &= \int_{0}^{t^{*}} \left|\varphi_{p}\left(\frac{t}{2t^{*}}\right)\right|^{p} dt + \int_{t^{*}}^{1} \left|\varphi_{p}\left(\frac{t+1-2t^{*}}{2(1-t^{*})}\right)\right|^{p} dt \\
&= 2t^{*} \int_{0}^{\frac{1}{2}}|\varphi_{p}(t)|^{p} dt + 2(1-t^{*}) \int_{\frac{1}{2}}^{1}|\varphi_{p}(t)|^{p} dt = t^{*} + 1 - t^{*} = 1
\end{align*}
and

\begin{align*}
& \int_{0}^{1} a^{p} [(u')^{+}]^{p} + b^{p} [(u')^{-}]^{p} dt \\
&= a^{p}\int_{0}^{t^{*}}\frac{1}{(2t^{*})^{p}}\left| \varphi_{p}'\left(\frac{t}{2t^{*}}\right)\right|^{p} dt + b^{p}\int_{t^{*}}^{1}\frac{1}{(2(1-t^{*}))^{p}}\left|\varphi_{p}'\left(\frac{t + 1- 2t^{*}}{2(1-t^{*})}\right)\right|^{p} dt \\
&= \frac{a^{p}}{(2t^{*})^{p-1}}\int_{0}^{\frac{1}{2}} |\varphi_{p}'(t)|^{p}\ dt + \frac{b^{p}}{(2(1-t^{*}))^{p-1}}\int_{\frac{1}{2}}^{1} |\varphi_{p}'(t)|^{p}\ dt \\
&= \frac{a^{p}}{(2t^{*})^{p-1}}\frac{\lambda_{1,p}(0,1)}{2}+\frac{b^{p}}{(2(1-t^{*}))^{p-1}}\frac{\lambda_{1,p}(0,1)}{2} \\
&=\frac{\lambda_{1,p}(0,1)}{2^{p}}\left(\frac{a^{p}}{\left(\frac{a}{a+b}\right)^{p-1}}+\frac{b^{p}}{\left(1-\frac{a}{a+b}\right)^{p-1}}\right)\\
&=\frac{\lambda_{1,p}(0,1)}{2^{p}}\left(a(a+b)^{p-1}+b(a+b)^{p-1}\right)\\
&= \frac{\lambda_{1,p}(0,1)}{2^{p}}(a+b)(a+b)^{p-1} \\
&= \left(\frac{a+b}{2}\right)^{p}\lambda_{1,p}(0,1).
\end{align*}
Consequently, $\lambda_{1,p}^{a,b}(0,1) \leq \left(\frac{a+b}{2}\right)^{p}\lambda_{1,p}(0,1)$, which contradicts \eqref{lambdaab2}. Therefore, $t_{0} = t_{1}$ and, from \eqref{lambdaab1}, we have $t_{0} = \frac{a}{a+b}$. Thus,

\[
\lambda_{1,p}^{a,b}(0,1) = a^p \lambda_{1,p}(0, 2t_0) = \frac{a^{p}}{\left(\frac{2a}{a+b}\right)^{p}}\lambda_{1,p}(0,1) = \left(\frac{a+b}{2}\right)^{p}\lambda_{1,p}(0,1).
\]
Furthermore, using strict convexity of the function (e.g. mimicking the proof of Theorem 3.1 in \cite{BFK})

\[
t \in \R \mapsto a^p\, (t^+)^p + b^p\, (t^-)^p,
\]
it is easily checked that there is at most one normalized extremizer with the same sign. This implies that $u$, as constructed in \eqref{lambdaab3}, is the desired positive eigenfunction $\varphi^{a,b,+}_p$.

Finally, the eigenspace of $\lambda^{a,b}_{1,p}(0,1)$ is formed by two collinear half-lines if and only if $\varphi_{p}^{a,b,+} = - \varphi_{p}^{a,b,-}$ in $(0,1)$. But, each of these functions has only a critical point, namely, $t_0$ and $1 - t_0$. Then, we have

\[
t_0 = \frac{a}{a + b} = \frac{1}{2},
\]
which is equivalent to $a = b$.

\section{On simplicity of zeroes related to $\Delta_{p}^{a,b}$}

This section is devoted to the proof of a result on simplicity of zeroes for nontrivial solutions of doubly asymmetric equations, both with respect to the operator and to nonlinearity.

\begin{teor} \label{T4} Let $a, b, L, \mu, \nu > 0$. If $u \in W_0^{1,p}(0,L)$ is a nontrivial weak solution of
\begin{equation*}
\begin{cases}
\begin{aligned}
& - \Delta_{p}^{a,b} u = \mu (u^+)^{p-1} - \nu (u^-)^{p-1} \ \ \text{in} \ \ (0,L), \\
& u(0) = u(L) = 0,
\end{aligned}
\end{cases}
\end{equation*}
then the zeroes of $u$ are simple.
\end{teor}

Before proving this theorem, we point out that eigenfunctions corresponding to the Dirichlet eigenvalues of $-\Delta_{p}^{a,b}$ other than $\lambda_{1,p}^{a,b}(0,L)$ necessarily change sign. Precisely, we have:

\begin{lema} \label{L2} Let $p > 1$ and $a, b, L > 0$. If $\lambda \neq \lambda_{1,p}^{a,b}(0,L)$ is a Dirichlet eigenvalue of $-\Delta_{p}^{a,b}$, then any eigenfunction associated to $\lambda$ changes sign.
\end{lema}

\n This is a quite standard result and its proof uses the well-known Picone's identity associated to the operator $-\Delta_{p}^{a,b}$, namely:

\begin{lema} \label{L3} Let $p > 1$, $a,b, L > 0$ and $H_{a,b}(t) := at^{+} +bt^{-}$ for $t \in \R$. Let also $u, v \in C^1(0,L)$ be functions such that $v > 0$ and $u \geq 0$ in $(0, L)$. Set

\[
R_{a,b}(u,v) := H_{a,b}^{p}(u') - \left(\frac{u^{p}}{v^{p-1}}\right)'H_{a,b}^{p-1}(v')H'_{a,b}(v')
\]
and

\[
P_{a,b}(u,v) := H_{a,b}^{p}(u') - \frac{u^{p}}{v^{p}}H_{a,b}^{p}(v') - p\frac{u^{p-1}}{v^{p-1}}H_{a,b}^{p}(v')\left(\frac{u'}{v'}-\frac{u}{v}\right).
\]
Then,

\[
R_{a,b}(u,v) = P_{a,b}(u,v) \geq 0 \ \ {\rm in}\ \ (0, L).
\]
\end{lema}

\begin{proof} The identity $R_{a,b}(u,v) = P_{a,b}(u,v)$ in $(0,L)$ follows from a straightforward computation. The nonnegativity of $P_{a,b}(u,v)$ in $(0,L)$ comes from the inequality
\[
H_{a,b}^{p}(t+h) > H_{a,b}^{p}(t) + p H_{a,b}^{p-1}(t)H'_{a,b}(t)h,\ \forall t, h \in \R,\ h \neq 0,
\]
which is an immediate consequence of the strict convexity of the function $t \in \R \mapsto H_{a,b}^{p}(t)$.
\end{proof}

\begin{proof}[Proof of Lemma \ref{L2}] We present the proof for completeness. Assume without loss of generality that $L = 1$. Let $u \in W_0^{1,p}(0,1)$ be a Dirichlet eigenfunction of $-\Delta_{p}^{a,b}$ associated to the eigenvalue $\lambda$. By elliptic regularity theory, one knows that $u$ belongs to $C^1[0,1]$. From the variational characterization of $\lambda_{1,p}^{a,b}(0,1)$, it is clear that $\lambda \geq \lambda_{1,p}^{a,b}(0,1)$.

Consider now the positive eigenfunction $\varphi_{p}^{a,b,+} \in C^1[0,1]$ corresponding to $\lambda_{1,p}^{a,b}(0,1)$ given in Theorem \ref{T1}. Assume by contradiction that $u$ is nonnegative in $(0,1)$. Then, the Harnack inequality guarantees that $u > 0$ in $(0,1)$. Using the fact that

\[
v_{\varepsilon} = \frac{(\varphi_{p}^{a,b,+})^{p}}{(u+\varepsilon)^{p-1}} \in W^{1,p}_{0}(0,1)
\]
for every $\varepsilon > 0$, by Lemma \ref{L3}, we have

\begin{align*}
0 & \leq \int_{0}^{1} P_{a,b}(\varphi_{p}^{a,b,+}, u + \varepsilon) \, dt = \int_{0}^{1} R_{a,b}(\varphi_{p}^{a,b,+},u+\varepsilon) \, dt  \\
& = \int_{0}^{1} H_{a,b}^{p}((\varphi_{p}^{a,b,+})') \, dt - \int_{0}^{1} \left( \frac{(\varphi_{p}^{a,b,+})^{p}}{(u+\varepsilon)^{p-1}} \right)' H_{a,b}^{p-1}((u + \varepsilon)')
H'_{a,b}((u + \varepsilon)') \, dt \\
& = \lambda_{1,p}^{a,b}(0,1) \int_{0}^{1}|\varphi_{p}^{a,b,+}|^{p} \, dt - \int_{0}^{1}\left(a^{p}[(u')^{+}]^{p-1}-b^{p}[(u')^{-}]^{p-1}\right) \left(\frac{(\varphi_{p}^{a,b,+})^{p}}{(u+\varepsilon)^{p-1}}\right)' \, dt \\
& = \lambda_{1,p}^{a,b}(0,1)\int_{0}^{1}|\varphi_{p}^{a,b,+}|^{p} \,dt - \lambda\int_{0}^{1}|u|^{p-2}u\left(\frac{(\varphi_{p}^{a,b,+})^{p}}{(u+\varepsilon)^{p-1}}\right) \, dt.
\end{align*}
Letting $\varepsilon \rightarrow 0$, we get

\[
0 \leq \left(\lambda_{1,p}^{a,b}(0,1) - \lambda\right)\int_{0}^{1}|\varphi_{p}^{a,b,+}|^{p} \,dt,
\]
so that $\lambda \leq \lambda_{1,p}^{a,b}(0,1)$, which leads us to the contradiction $\lambda = \lambda_{1,p}^{a,b}(0,1)$.
\end{proof}

We now prove Theorem \ref{T4} through an inductive procedure with the aid of Lemma \ref{L2} and maximum principles.

\begin{proof}[Proof of Theorem \ref{T4}] By elliptic regularity theory, we have $u \in C^1[0,L]$. Since $u$ is nonzero, there is $t_{0} \in (0, L)$ such that $u(t_{0}) \neq 0$. Assume that $u(t_{0}) > 0$ (the negative case is similar). Let $(\alpha_{0},\beta_{0}) \subset (0,L)$ be the maximal interval containing $t_{0}$ where $u$ is positive. Clearly, $u \in C^1[\alpha_{0},\beta_{0}]$ is a positive weak solution of

\begin{equation}\label{mu1}
\begin{cases}
\begin{aligned}
& - \Delta_{p}^{a,b} u = \mu|u|^{p-2}u \ \ \text{in} \ \ (\alpha_{0},\beta_{0}), \\
& u(\alpha_{0}) = u(\beta_{0}) = 0.
\end{aligned}
\end{cases}
\end{equation}
Since $\mu > 0$, by the Hopf's lemma, $u$ satisfies $u'(\alpha_{0}) > 0$ and $u'(\beta_{0}) < 0$. If $(\alpha_{0},\beta_{0}) = (0,L)$ then we have nothing to do. Otherwise, if $\alpha_{0} > 0$ then there is an interval $(\alpha_{1},\alpha_{0}) \subset (0,L)$ such that $u$ is negative in it and $u(\alpha_{0}) = u(\alpha_{1}) = 0$. Similarly, if $\beta_{0} < L$ then we have an interval $(\beta_{0},\beta_{1})$ such that $u$ is negative in it and $u(\beta_{0}) = u(\beta_{1}) = 0$. In each case, we have that $u$ is a negative weak solution of the equations

\begin{equation}\label{nu1}
\begin{cases}
\begin{aligned}
& - \Delta_{p}^{a,b} u = \nu|u|^{p-2}u \ \ \text{in} \ \ (\alpha_{1},\alpha_{0}), \\
& u(\alpha_{0}) = u(\alpha_{1}) = 0,
\end{aligned}
\end{cases}
\end{equation}
\begin{equation*}
\begin{cases}
\begin{aligned}
& - \Delta_{p}^{a,b} u = \nu|u|^{p-2}u \ \ \text{in} \ \ (\beta_{0},\beta_{1}), \\
& u(\beta_{0}) = u(\beta_{1}) = 0.
\end{aligned}
\end{cases}
\end{equation*}
Since $(\alpha_{1},\beta_{1}) \neq (0,L)$ and $\nu > 0$ we can keep doing this argument and, with the aid of the Hopf's lemma, the sign of $u$ follows interchanging. By Lemma \ref{L2}, notice that due to $u$ be a nonzero weak solution of \eqref{mu1} and \eqref{nu1} that does not change sign, then

\[
\mu = \lambda_{1,p}^{a,b}(\alpha_{0},\beta_{0})\ \ \ {\rm and}\ \ \ \nu = \lambda_{1,p}^{a,b}(\alpha_{1},\alpha_{0}).
\]
This shows that all intervals where $u$ is positive have length $\beta_{0}-\alpha_{0}$ and all intervals where $u$ is negative have length $\alpha_{0}-\alpha_{1}$. Therefore, if we keep repeating this argument, there will be positive integers $m$ and $n$ such that

\[
0 \leq \alpha_{m} < \begin{cases}
\begin{aligned}
& \beta_{0}-\alpha_{0}, \ \text{if $m$ is odd} \\
& \alpha_{0}-\alpha_{1}, \ \text{if $m$ is even}
\end{aligned}
\end{cases}
\]
and
\[
0 \leq L-\beta_{n} < \begin{cases}
\begin{aligned}
& \beta_{0}-\alpha_{0}, \ \text{if $n$ is even} \\
& \alpha_{0}-\alpha_{1}, \ \text{if $n$ is odd}.
\end{aligned}
\end{cases}
\]
Assume first $m$ is odd and $\alpha_{m} > 0$. As argued before we will have $u'(\alpha_{m}) < 0$, which implies there is an interval $(\alpha_{m+1},\alpha_{m})$ with $\alpha_{m+1} \geq 0$ such that $u$ is a positive weak solution of

\begin{equation}
\begin{cases}
\begin{aligned}
& - \Delta_{p}^{a,b} u = \mu|u|^{p-2}u \ \ \text{in} \ \ (\alpha_{m + 1},\alpha_{m}), \\
& u(\alpha_{m+1}) = u(\alpha_{m}) = 0.
\end{aligned}
\end{cases}
\end{equation}
Therefore, since $\mu > 0$, by Lemma \ref{L2}, we have $\mu = \lambda_{1,p}^{a,b}(\alpha_{m+1},\alpha_{m})$. Hence,

\[
\beta_{0}-\alpha_{0} = \alpha_{m}-\alpha_{m+1} \leq \alpha_{m} < \beta_{0}-\alpha_{0},
\]
leading to a contradiction, and thus $\alpha_{m} = 0$. The argument for $\beta_{n} = L$ when $n$ is odd and also for the case $m$ or $n$ even is analogous. This proves that the step-by-step procedure performed above drops out after a finite number of steps, and so the zeroes of $u$ in $[0,L]$ are precisely given by $\{\alpha_{i}\}_{i=0}^{m}$ and $\{\beta_{i}\}_{i=0}^{n}$ and the derivative at each of them is nonzero due to the Hopf's lemma. This concludes the proof.
\end{proof}

\section{Proof of Theorem \ref{T3}}

The proof of this theorem uses the characterization provided in Theorem \ref{T1} and the zeroes simplicity property ensured in Theorem \ref{T4}.

\begin{proof}[Proof of Theorem \ref{T3}] Let $u \in W_0^{1,p}(0,L)$ be a nontrivial weak solution of \eqref{FS} corresponding to a couple $(\mu, \nu) \in \Sigma^{a,b}_p(0,L)$. By Theorem \ref{T1} and Lemma \ref{L2}, we have three possible cases:

\[
\mu = \left( \frac{a + b}{2} \right)^p \lambda_{1,p}(0,L)\ \ {\rm or}\ \ \nu = \left( \frac{a + b}{2} \right)^p \lambda_{1,p}(0,L)\ \ {\rm or}\ \ (\mu, \nu) \in \R_+^2.
\]
From now on we assume the latter one. By Theorem \ref{T4}, since the zeroes of $u$ are simple, there are positive integers $P = P(\mu, \nu)$ and $N = N(\mu, \nu)$ and an interspersed cover of $[0,L]$ by open intervals $\{I_i\}^P_{i=1}$ and $\{J_j\}^N_{j=1}$ such that $u$ is positive in $I_i$ and negative in $J_j$ and satisfies the equations in the weak sense

\begin{equation} \label{P}
\begin{cases}
\begin{aligned}
- \Delta_{p}^{a,b} u & = \mu \vert u \vert^{p-2} u \ \  \text{in} \ I_i, \\
u & = 0  \hspace{1.5cm} \text{on} \ \partial I_i
\end{aligned}
\end{cases}
\end{equation}
and

\begin{equation} \label{N}
\begin{cases}
\begin{aligned}
- \Delta_{p}^{a,b} u & = \nu \vert u \vert^{p-2} u  \ \ \text{in} \ J_j, \\
u & = 0 \hspace{1.5cm} \text{on} \ \partial J_j.
\end{aligned}
\end{cases}
\end{equation}
Since $\mu, \nu > 0$, by Theorem \ref{T1} and Lemma \ref{L2}, we have

\[
\mu = \lambda_{1,p}^{a,b}(I_i) = |I_i|^{-p} \lambda_{1,p}^{a,b}(0,1)
\]
and

\[
\nu = \lambda_{1,p}^{a,b}(J_j) = |J_j|^{-p} \lambda_{1,p}^{a,b}(0,1).
\]
Hence, all intervals $I_i$ has the same measure $l_\mu$, just as all $J_j$ has the same measure $l_\nu$, where

\[
l_\mu = \left( \frac{\lambda^{a,b}_{1,p}(0,1)}{\mu} \right)^{1/p}\ \ \ {\rm and}\ \ \ l_\nu = \left( \frac{\lambda^{a,b}_{1,p}(0,1)}{\nu} \right)^{1/p}.
\]
Since the cover of $[0, L]$ is interspersed, we have $\vert P-N \vert \in \{0,1\}$ and $P l_{\mu} + N l_{\nu} = L$. In other words,

\[
P \frac{\left(\lambda_{1,p}^{a,b}(0,1)\right)^{\frac{1}{p}}}{\mu^{\frac{1}{p}}} + N \frac{\left(\lambda_{1,p}^{a,b}(0,1)\right)^{\frac{1}{p}}}{\nu^{\frac{1}{p}}} = L,
\]
or equivalently,

\[
\frac{P}{\mu^{\frac{1}{p}}} + \frac{N}{\nu^{\frac{1}{p}}} = \left(\lambda_{1,p}^{a,b}(0,L)\right)^{-\frac{1}{p}}.
\]
Consequently, if we set $S \coloneqq \left\{(P,N) \in \N \times \N \colon\, \vert P-N \vert \in \{0,1\}\right\}$ where $\N$ represents the set of positive integers, then

\begin{align*}
\Sigma_{p}^{a,b}(0,L) \cap \R_+^2 &= \bigcup_{(P,N) \in S}\left\{(\mu,\nu) \in \R_+^2 \colon \frac{P}{\mu^{\frac{1}{p}}} + \frac{N}{\nu^{\frac{1}{p}}} = \left(\lambda_{1,p}^{a,b}(0,L)\right)^{-\frac{1}{p}} \right\} \\
&= \bigcup_{(P,N) \in S}\left\{(\mu,\nu) \in \R_+^2 \colon \frac{P}{\mu^{\frac{1}{p}}} + \frac{N}{\nu^{\frac{1}{p}}} =  \left(\frac{a+b}{2}\right)^{-1} \left(\lambda_{1,p}(0,L)\right)^{-\frac{1}{p}} \right\} \\
&= \left(\frac{a+b}{2}\right)^{p}\bigcup_{(P,N) \in S}\left\{(\mu,\nu) \in \R_+^2 \colon \frac{P}{\mu^{\frac{1}{p}}} + \frac{N}{\nu^{\frac{1}{p}}} =  \left(\lambda_{1,p}(0,L)\right)^{-\frac{1}{p}}\right\} \\
&= \left(\frac{a+b}{2}\right)^{p}\Sigma_{p}(0,L) \cap \R_+^2.
\end{align*}

Finally, we characterize the nontrivial solution $u \in W_0^{1,p}(0,L)$ of \eqref{FS}. By elliptic regularity theory, we know that $u \in C^1[0,L]$. Let $\gamma_i$ and $\pi_j$ be the increasing affine bijections from $\bar{I}_i$ and $\bar{J}_j$ onto $[0,1]$, respectively. Since $u$ is a positive weak solution of \eqref{P} in $I_1$ and a negative weak solution of \eqref{N} in $J_1$, by Theorem \ref{T1}, we have $u = c^+ \varphi_{p}^{a,b,+} \circ \gamma_1$ in $I_1$ and $u = c^- \varphi_{p}^{a,b,-} \circ \pi_1$ in $J_1$ for some positive constants $c^+$ and $c^-$. Noting that $u$ is of $C^1$ class at the intersection $\overline{I}_1 \cap \overline{J}_1$, then there is a constant $c > 0$ such that $c^+ = c\, l_\mu$ and $c^- = c\, l_\nu$. Repeating the above argument for all pairs of intervals $I_i$ and $J_j$ such that $\overline{I}_i \cap \overline{J}_j \neq \emptyset$, taking into account that $u \in C^1[0,L]$ and that $\{I_i\}^P_{i=1}$ and $\{J_j\}^N_{j=1}$ cover interspersedly the interval $[0,L]$, we easily deduce that $u = c\, l_\mu\, \varphi_{p}^{a,b,+} \circ \gamma_i$ in $I_i$ and $u =  c\, l_\nu\,  \varphi_{p}^{a,b,-} \circ \pi_j$ in $J_j$ for the same constant $c$ above, independent of $i$ and $j$. Therefore, it follows that $u = c \Phi_{(\mu, \nu), p, L}^{a,b}$ in $[0, L]$ and the proof is ended.
\end{proof}

\appendix
\section{An auxiliary result on glued solutions}

In this appendix we provide a proof of the following lemma:

\begin{lema} \label{L1} Let $(\alpha,\beta) \subset (0, L)$ be two nonempty open intervals, let $f : \R \rightarrow \R$ be a continuous function and let $u \in C^{1}[0,L]$ be a weak solution of

\begin{equation*}
\begin{cases}
\begin{aligned}
& - \Delta_{p}^{a,b} u = f(u) \ \text{in} \ (0,L), \\
& u(0) = u(L) = 0.
\end{aligned}
\end{cases}
\end{equation*}
If $u(\alpha) = 0$, $u'(\beta) = 0$ and $u' \geq 0$ in $(\alpha,\beta)$, then the function

\begin{equation*}
w(t) =
\begin{cases}
\begin{aligned}
     & u(t), \ \text{if} \ t \in [\alpha,\beta), \\
     & u(2\beta - t), \ \text{if} \ t \in [\beta,2\beta - \alpha]
\end{aligned}
\end{cases}
\end{equation*}
belongs to $C^1[\alpha, 2 \beta - \alpha]$ and is a weak solution of

\begin{equation*}
\begin{cases}
\begin{aligned}
& -a^{p}\Delta_{p} w = f(w) \ \text{in} \ (\alpha, 2\beta - \alpha), \\
& w(\alpha) = w(2\beta - \alpha) = 0.
\end{aligned}
\end{cases}
\end{equation*}
\end{lema}

\begin{proof} Clearly, we have $w \in C^1[\alpha, 2 \beta - \alpha]$. Take any $\varphi \in W^{1,p}_{0}(\alpha, 2\beta - \alpha)$ and consider the following functions for each $\varepsilon > 0$:

\begin{equation*}
\phi_{\varepsilon}(t) =
\begin{cases}
\begin{aligned}
     & \varphi(t) \ \ \text{if} \ t \in (\alpha, \beta - \varepsilon), \\
     & -\frac{\varphi(\beta - \varepsilon)}{\varepsilon}(t - \beta) \ \ \text{if} \ t \in [\beta - \varepsilon, \beta), \\
     & 0 \ \ \text{if} \ t \in \mathbb{R} \setminus (\alpha, \beta),
\end{aligned}
\end{cases}
\end{equation*}
\begin{equation*}
\psi_{\varepsilon}(t) =
\begin{cases}
\begin{aligned}
     & \varphi(2\beta - t) \ \ \text{if} \ t \in (\alpha, \beta - \varepsilon), \\
     & -\frac{\varphi(\beta + \varepsilon)}{\varepsilon}(t-\beta) \ \ \text{if} \ t \in [\beta - \varepsilon, \beta), \\
     & 0 \ \ \text{if} \ t \in \mathbb{R} \setminus (\alpha, \beta).
\end{aligned}
\end{cases}
\end{equation*}
It is easy to check that $\phi_{\varepsilon},\psi_{\varepsilon} \in W^{1,p}_0(0,L)$ and that

\begin{eqnarray*}
&& \lim_{\varepsilon \to 0^{+}} \phi_{\varepsilon}(t) + \psi_{\varepsilon}(2\beta - t) = \varphi(t) \ \text{for every} \ t \in (\alpha, 2\beta - \alpha),\\
&& \lim_{\varepsilon \to 0^{+}} \phi_{\varepsilon}'(t) - \psi_{\varepsilon}'(2\beta - t) = \varphi'(t) \ \text{for almost every} \ t \in (\alpha, 2\beta - \alpha).
\end{eqnarray*}
So, by the dominated convergence theorem, we derive

\begin{equation}\label{limu+1}
\lim_{\varepsilon \to 0^{+}} \int_{\alpha}^{2\beta - \alpha}f(w)\left(\phi_{\varepsilon}(t) + \psi_{\varepsilon}(2\beta - t)\right)\, dt = \int_{\alpha}^{2\beta - \alpha}f(w)\varphi \, dt.
\end{equation}
Also, since $-\Delta_{p}^{a,b}u = -a^{p}\Delta_{p}u$ in $(\alpha,\beta)$, we have

\[
\int_{\alpha}^{\beta} a^{p}|w'|^{p-2} w'\phi_{\varepsilon}' \, dt = \int_{\alpha}^{\beta - \varepsilon} a^{p}|u'|^{p-2}u'\varphi' \, dt + \int_{\beta - \varepsilon}^{\beta} a^{p}|u'|^{p-2}u'\left(-\frac{\varphi(\beta - \varepsilon)}{\varepsilon}\right) \, dt
\]
and since $u'$ is continuous in $[\alpha, \beta]$ and $u'(\beta) = 0$, we also get

\[
\left|\int_{\beta - \varepsilon}^{\beta} a^{p}|u'|^{p-2}u'\left(-\frac{\varphi(\beta - \varepsilon)}{\varepsilon}\right) \, dt\right| \leq a^{p} \lVert \varphi\rVert_{\infty}\lVert u' \rVert_{L^{\infty}[\beta - \varepsilon,\beta]}^{p-1} \xrightarrow[\varepsilon \to 0^{+}]{} 0.
\]
Consequently,

\begin{equation}\label{limu+2}
\lim_{\varepsilon \to 0^{+}} \int_{\alpha}^{\beta} a^{p} |w'|^{p-2} w'\phi_{\varepsilon}' \, dt = \int_{\alpha}^{\beta} a^{p} |w'|^{p-2} w'\varphi' \, dt.
\end{equation}
Now notice that

\begin{align*}
\int_{\beta}^{2\beta - \alpha} a^{p} |w'(t)|^{p-2} w'(t) \left(-\psi_{\varepsilon}'(2\beta - t)\right) \, dt &= a^{p} \int_{\beta}^{2\beta - \alpha} |u'(2\beta - t)|^{p-2}(-u'(2\beta - t))\left(-\psi_{\varepsilon}'(2\beta - t)\right) \, dt \\
&= a^{p}\int_{\alpha}^{\beta} |u'(t)|^{p-2} u'(t) \psi_{\varepsilon}'(t) \, dt \\
&= a^{p}\int_{\alpha}^{\beta - \varepsilon} |u'(t)|^{p-2} u'(t)\varphi'(2\beta - t) \, dt \\
&\ \ \ + a^{p}\int_{\beta - \varepsilon}^{\beta} |u'(t)|^{p-2} u'(t) \left(-\frac{\varphi(\beta + \varepsilon)}{\varepsilon}\right) \, dt \\
&=\int_{\beta}^{2\beta - \alpha} a^{p}|w'(t)|^{p-2} w'(t)\varphi'(t) \, dt \\
&\ \ \ + a^{p}\int_{\beta - \varepsilon}^{\beta} |u'(t)|^{p-2} u'(t)\left(-\frac{\varphi(\beta + \varepsilon)}{\varepsilon}\right) \, dt
\end{align*}
and, arguing as before, we have

\begin{equation}\label{limu+3}
\lim_{\varepsilon \to 0^{+}} \int_{\beta}^{2\beta - \alpha} a^{p} |w'(t)|^{p-2} w'(t)\left(-\psi_{\varepsilon}'(2\beta - t)\right) \, dt = \int_{\beta}^{2\beta - \alpha} a^{p} |w'|^{p-2} w'\varphi' \, dt.
\end{equation}
Using \eqref{limu+1}, \eqref{limu+2}, \eqref{limu+3} and the fact that $\phi_{\varepsilon}$ and $\psi_{\varepsilon}$ are test functions, we obtain

\begin{align*}
\int_{\alpha}^{2\beta - \alpha} f(w)\varphi \, dt &= \lim_{\varepsilon \to 0^{+}} \int_{\alpha}^{2\beta - \alpha} f(w) \left(\phi_{\varepsilon}(t) + \psi_{\varepsilon}(2\beta - t)\right) \, dt \\
&= \lim_{\varepsilon \to 0^{+}} \int_{\alpha}^{\beta}f(u)\phi_{\varepsilon} \, dt + \int_{\beta}^{2\beta - \alpha} f(u(2\beta - t))\psi_{\varepsilon}(2\beta - t) \, dt \\
&= \lim_{\varepsilon \to 0^{+}}\int_{\alpha}^{\beta} a^{p}|w'|^{p-2} w' \phi_{\varepsilon}' \, dt + \int_{\alpha}^{\beta} f(u) \psi_{\varepsilon} \, dt \\
&= \lim_{\varepsilon \to 0^{+}}\int_{\alpha}^{\beta} a^{p} |w'|^{p-2} w' \phi_{\varepsilon}' \, dt + \int_{\alpha}^{\beta} a^{p}|u'|^{p-2}u'\psi_{\varepsilon}' \, dt \\
&= \lim_{\varepsilon \to 0^{+}}\int_{\alpha}^{\beta} a^{p}|w'|^{p-2} w' \phi_{\varepsilon}' \, dt - \int_{\beta}^{2\beta - \alpha} a^{p} |w'(t)|^{p-2} w'(t) \psi_{\varepsilon}'(2\beta - t) \,dt \\
&= \int_{\alpha}^{\beta} a^{p}|w'|^{p-2} w'\varphi' \, dt + \int_{\beta}^{2\beta - \alpha} a^{p}|w'|^{p-2} w' \varphi' \, dt \\
&= \int_{\alpha}^{2\beta - \alpha} a^{p} |w'|^{p-2} w'\varphi' \, dt.
\end{align*}
\end{proof}

\subsection*{Acknowledgement}

The first author was partially supported by CAPES (PROEX 88887.712161/2022-00) and the second author was partially supported by CNPq (PQ 307432/2023-8, Universal 404374/2023-9) and Fapemig (Universal-APQ 01598-23).

\end{document}